\documentclass[a4paper,11pt]{article}

\usepackage[a4paper]{geometry}

\usepackage{authblk}

\usepackage{amsmath}
\usepackage{amsthm}
\usepackage{amssymb}
\usepackage{amsfonts}
\usepackage{mathtools}
\usepackage{bm}
\usepackage{bussproofs}

\usepackage{cancel}
\usepackage{lscape}

\usepackage{float}
\usepackage{booktabs}
\usepackage{caption}
\usepackage{adjustbox}

\usepackage{graphicx}
\usepackage{tikz}
\usetikzlibrary{cd,backgrounds,positioning,trees,shapes,arrows,patterns,topaths,calc}

\usepackage[
  backend=biber,
  style=numeric,
  natbib=true,
  sorting=nty,
  giveninits=true,
  maxbibnames=99,
  abbreviate=false,
  useprefix=false,
  isbn=false
]{biblatex}

\usepackage{xurl}
\usepackage{hyperref}
\hypersetup{
  hidelinks,
  linktoc = all,
  breaklinks = true,
}
\usepackage{aliascnt}
\usepackage[nameinlink]{cleveref}

\newtheorem{theorem}{Theorem}
\numberwithin{theorem}{section}

\newaliascnt{lemma}{theorem}
\newtheorem{lemma}[lemma]{Lemma}
\aliascntresetthe{lemma}

\newaliascnt{corollary}{theorem}

\aliascntresetthe{corollary}

\newaliascnt{fact}{theorem}

\aliascntresetthe{fact}

\newaliascnt{question}{theorem}

\aliascntresetthe{question}

\newaliascnt{proposition}{theorem}
\newtheorem{proposition}[proposition]{Proposition}
\aliascntresetthe{proposition}

\Crefname{theorem}{Theorem}{Theorems}
\Crefname{lemma}{Lemma}{Lemmas}
\Crefname{corollary}{Corollary}{Corollaries}
\Crefname{fact}{Fact}{Facts}
\Crefname{question}{Question}{Questions}
\Crefname{proposition}{Proposition}{Propositions}

\theoremstyle{definition}

\newaliascnt{definition}{theorem}
\newtheorem{definition}[definition]{Definition}
\aliascntresetthe{definition}

\newaliascnt{remark}{theorem}
\newtheorem{remark}[remark]{Remark}
\aliascntresetthe{remark}

\newaliascnt{example}{theorem}

\aliascntresetthe{example}

\newaliascnt{notation}{theorem}

\aliascntresetthe{notation}

\Crefname{definition}{Definition}{Definitions}
\Crefname{remark}{Remark}{Remarks}
\Crefname{example}{Example}{Examples}
\Crefname{notation}{Notation}{Notations}

\theoremstyle{remark}

\newaliascnt{claim}{theorem}
\newtheorem{claim}[claim]{Claim}
\aliascntresetthe{claim}

\Crefname{claim}{Claim}{Claims}

\renewcommand{\sf}[1]{\mathsf{#1}}

\newcommand{\alg}[1]{\mathfrak{#1}}
    \newcommand{\A}{\alg{A}}

\renewcommand{\sp}[1]{\mathfrak{#1}}
    \newcommand{\X}{\sp{X}}
    \newcommand{\Y}{\sp{Y}}

\newcommand{\class}[1]{\mathcal{#1}}

    \newcommand{\J}{\class{J}}

    \newcommand{\KF}{\sf{KF}}

\newcommand{\fin}{_\mathrm{fin}}

\newcommand{\pow}[1]{\mathcal{P} (#1)}

\newcommand{\logic}[1]{\mathsf{#1}}

    \newcommand{\K}{\logic{K}}
    \newcommand{\Kf}{\logic{K4}}
    \newcommand{\Kff}[2]{\logic{K4}^{#1}_{#2}}
    \newcommand{\Sf}{\logic{S4}}

\newcommand{\NExt}[1]{\mathop{\mathsf{NExt}}{#1}}

\renewcommand{\sf}[1]{\mathsf{#1}}
\renewcommand{\P}{\mathsf{P}}

\newcommand{\Val}{\mathsf{Val}}

\renewcommand{\phi}{\varphi}
\newcommand{\emp}{\emptyset}

\newcommand{\Dia}{\Diamond}

\newcommand{\CSt}{\mathsf{CSt}}

\newcommand{\Mod}{\mathsf{Mod}}

\renewcommand{\L}{\mathcal{L}}

\newcommand{\NP}{\mathsf{NP}}

\newcommand{\DLOGTIME}{\mathsf{DLOGTIME}}
\newcommand{\ACz}{\mathsf{AC}^0}

\renewcommand{\sc}[1]{\textup{\textsc{#1}}}

\renewcommand{\Val}[1]{\sc{Val(#1)}}
\newcommand{\Parity}{\sc{Parity}}
\newcommand{\EvenParity}{\sc{Even-Parity}}

\begin{document}

\title{Elementary modal logics and circuit complexity}
\author{Tenyo Takahashi}
\affil{\small Institute for Logic, Language and Computation, University of Amsterdam \\ \href{mailto:t.takahashi@uva.nl}{t.takahashi@uva.nl}}
\date{}
\maketitle

\begin{abstract}
A modal logic is elementary if it is sound and complete with respect to an elementary class of Kripke frames. We prove that if a finitely axiomatizable modal logic $L$ is elementary, then the validity problem for $L$, which asks whether a given finite Kripke frame validates $L$, is in the circuit complexity class uniform $\ACz$. As an application, we construct infinitely many finitely axiomatizable $\Kf$-stable logics that are not elementary. This result provides a negative answer to Problem 6.4 in [Bezhanishvili et al., 2018] for $\Kf$-stable logics.
\end{abstract}

% keywords: modal logic, stable logic, elementarity, computational complexity, uniform AC0

\section{Introduction}
This paper studies the connection between elementarity of modal logics and computational complexity. 

A modal logic $L$ is \emph{elementary}, also called \emph{elementarily determined} or \emph{first-order complete}, if $L$ is sound and complete with respect to an elementary class of Kripke frames. That is, there is a class $\class{K}$ of Kripke frames defined by a set of first-order sentences in the language of Kripke frames such that, for any modal formula $\phi$, it holds that $\phi \in L$ iff $\phi$ is valid on all frames in $\class{K}$. Elementarity has been extensively studied in modal logic (see, e.g., \cite[Chapter 10]{czModalLogic1997} and \cite{valentingorankoModelTheoryModal2007}), both as a route to Kripke completeness and as a way to understand the interplay between first-order logic and modal logic; modal frame validity is intrinsically second-order. Many well-known modal logics are elementary, such as $\K$, $\Kf$, and $\Sf$. In fact, standard proofs of Kripke completeness via \emph{canonical frames} (see, e.g., \cite[Chapter 5]{czModalLogic1997} and \cite[Chapter 4]{blackburnModalLogic2001}) often make use of the corresponding first-order frame conditions, and thus prove the elementarity at the same time. Behind this phenomenon is the celebrated Fine's theorem \cite{ConnectionsElementaryModal1975}, which states that every elementary logic is \emph{canonical}, that is, it is valid on its canonical frame, which in turn implies Kripke completeness.

\emph{Stable logics} were introduced in \cite{stablecanonicalrules,stablemodallogic} as a filtration counterpart of \emph{subframe logics} \cite{fineK4II}, which are based on the idea of selective filtration. Let $M$ be a modal logic. A logic $L \supseteq M$ is \emph{$M$-stable} if $L$ is sound and complete with respect to a class of modal spaces validating $M$ which satisfies a certain closure property (see the beginning of Section \ref{Sec 4} for a precise definition). All $M$-stable logics have the finite model property if $M$ admits filtration \cite{stablemodallogic}, or more weakly, if $M$ admits definable filtration, a generalized version of the standard filtration \cite{takahashi2025stablecanonicalrulesformulas}. This result applies to, for example, $M = \K$, $\Kf$, and $\Sf$. The finite model property of such $M$-stable logics is parallel to the fact that all transitive subframe logics have the finite model property \cite{fineK4II}. On the other hand, while $\logic{Grz}$ is a well-known example of a non-elementary subframe logic (see, e.g., \cite[Theorem 6.8]{czModalLogic1997}), it remains open whether every $M$-stable logic is elementary, for reasonable $M$ such as $\K$, $\Kf$, or $\Sf$ \cite[Problem 6.4]{stablemodallogic} (see also \cite[Problem 4]{FiltrationRevisitedLattices2018}).

The elementarity of a logic $L$ should not be confused with the strictly stronger condition that the class $\KF(L)$ of all Kripke frames validating $L$ is elementary. Classical examples of an elementary logic with a non-elementary frame class include the logic $\logic{KMT}$ studied by Hughes \cite{hughesEveryWorldCan1990}, and the McKinsey-Lemmon logic $\logic{KM^\infty}$ originally defined by Lemmon \cite{LemmonNotes} and studied also in \cite{goldblattMcKinseyLemmonLogic2007}. Since elementarity concerns the existence of a suitable class of Kripke frames, proving a logic $L$ is not elementary requires showing that $L$ is sound and complete with respect to \emph{no} elementary class of Kripke frames, which is often challenging.

This paper introduces a method to prove the non-elementarity for modal logics by examining the computational complexity of their validity problem over finite Kripke frames. More specifically, for a logic $L$, let \Val{$L$} be the problem of deciding whether a given finite Kripke frame validates $L$. We show that for a finitely axiomatizable elementary logic $L$, the validity problem \Val{$L$} is in uniform $\ACz$. Uniform $\ACz$ is a very low circuit complexity class that is strictly contained in \textsf{LSPACE}; see \Cref{Sec 2.2} for a precise definition. Thus, we can prove that a finitely axiomatizable logic $L$ is not elementary by showing that \Val{$L$} is not in uniform $\ACz$. 

The proof combines an intermediate result in \cite{GoldblattHodkinsonVenema2004Erdos} used to disprove the converse of Fine's theorem with a compactness argument using the method of standard translations (see, e.g., \cite[Section 2.4]{blackburnModalLogic2001}). Intuitively, for an elementary logic $L$, the former allows us to find a ``nice'' elementary class $\class{K}$ that contains all finite frames validating $L$, and the latter allows us to take the class $\class{K}$ to be even nicer, in the sense that it is defined by finitely many first-order sentences, provided that $L$ is finitely axiomatizable. Since the truth of a fixed first-order sentence can be decided in uniform $\ACz$, we conclude that \Val{$L$} is in uniform $\ACz$ for any finitely axiomatizable elementary logic $L$. This yields a necessary condition for a finitely axiomatizable logic to be elementary. This approach of showing the non-elementarity, as we will see, does not involve a complicated analysis of the canonical frame or canonical extensions of modal algebras.

As an application, we prove that there are infinitely many finitely axiomatizable non-elementary $\Kf$-stable logics. Since $\Kf$-stable logics are axiomatized by \emph{stable formulas} \cite{stablemodallogic}, whose validity has a semantic characterization on Kripke frames, we can turn the validity of $\Kf$-stable logics into combinatorial conditions between Kripke frames. This allows us to reduce $\EvenParity$, a typical problem outside uniform $\ACz$ (even $\ACz$), to the validity problem of some $\Kf$-stable logics, which are thus not elementary. This result answers \cite[Problem 6.4]{stablemodallogic} (see also \cite[Problem 4]{FiltrationRevisitedLattices2018}) negatively for $\Kf$-stable logics. 

The paper is organized as follows. We recall preliminaries on modal logic and circuit complexity in \Cref{Sec 2}. In \Cref{Sec 3}, we prove the necessary condition for a finitely axiomatizable logic to be elementary, namely, that the problem \Val{$L$} is in uniform $\ACz$. In \Cref{Sec 4}, we show the existence of infinitely many non-elementary $\Kf$-stable logics. 

\section{Preliminaries} \label{Sec 2}

\subsection{Modal logic} \label{Sec 2.1}

We assume familiarity with basic modal logic and Kripke semantics, and use \cite{czModalLogic1997,blackburnModalLogic2001} as our main references. We only work with normal unimodal logics, so we simply call them logics. For a logic $L$, let $\KF(L)$ be the class of all Kripke frames validating $L$. For a class $\class{K}$ of Kripke frames, let $\class{K}\fin$ denote the class of its finite members. The drawings of Kripke frames follow the convention that $\bullet$ denotes an irreflexive point and $\circ$ denotes a reflexive point. 

Let $\L$ be the first-order language of Kripke frames, that is, the language with a binary relation symbol $R$ and the equality. Kripke frames are naturally regarded as $\L$-structures. We use the same symbol $\models$ for both modal validity and first-order satisfaction, assuming no confusion arises. A class of Kripke frames is \emph{elementary} if it is defined by a (possibly infinite) set of first-order sentences in $\L$. A logic $L$ is \emph{elementary} if it is sound and complete with respect to an elementary class of Kripke frames. Note that this does not require the class $\KF(L)$ itself to be elementary. 

Let us recall some basics of the correspondence between modal algebras (or Boolean algebras with operators) and Kripke frames, in the setting of unimodal logics; see, e.g., \cite[Chapter 5]{blackburnModalLogic2001} and \cite{venema6AlgebrasCoalgebras2007} for details. A \emph{modal algebra} is a pair $\A = (A, \Dia)$ of a Boolean algebra $A$ and a unary operation $\Dia$ on $A$ such that $\Dia 0 = 0$ and $\Dia (a \lor b) = \Dia a \lor \Dia b$. Valuations, satisfaction, and validity are defined as usual. For a modal algebra $\A = (A, \Dia)$, let $\A_+$ be the \emph{canonical structure}\footnote{We are inclined to call $\A_+$ the canonical frame of $\A$. But, unfortunately, this term is generally reserved for the \emph{canonical frame} of a logic. Thus, not to cause confusion, we follow \cite{GoldblattHodkinsonVenema2004Erdos} to call $\A_+$ the canonical structure of $\A$.} of $\A$, that is, $\A_+$ is the Kripke frame $(U(A), R_{\Dia})$, where $U(A)$ is the set of all ultrafilters of $A$ and $R_{\Dia}$ is defined by $u R_{\Dia} v$ iff $a \in v$ implies $\Dia a \in u$ for all $a \in A$. Another way to describe the canonical structure $\A_+$ is that $\A_+$ is the underlying Kripke frame of the dual \emph{modal space} (or \emph{descriptive frame}) of $\A$ (see, e.g., \cite[Section 8.4]{czModalLogic1997} and \cite[Section 5.5]{blackburnModalLogic2001}). Conversely, for a Kripke frame $F = (F, R)$, let $F^+$ be the \emph{complex algebra} of $F$, that is, $F^+$ is the modal algebra $(\pow{F}, {\Dia}_R)$, where $\pow{F}$ is the Boolean algebra of all subsets of $F$ and ${\Dia}_R$ is defined by ${\Dia}_R(a) = R^{-1}[a]$. It is easy to see that $\A_+ \models \phi$ implies $\A \models \phi$ and $F^+ \models \phi$ iff $F \models \phi$. The algebra $(\A_+)^+$ is called the \emph{canonical extension} of $\A$, and the Kripke frame $(F^+)_+$ is called the \emph{ultrafilter extension} of $F$ (see also \cite[Section 2.5]{blackburnModalLogic2001}). For a logic $L$, following the notation in \cite{GoldblattHodkinsonVenema2004Erdos}, we write $\CSt(L) = \{\A_+: \A \models L\}$ and $\KF(L) = \{F: F^+ \models L\} = \{F: F \models L\}$. A logic $L$ is \emph{canonical} if $\CSt(L) \subseteq \KF(L)$.

\subsection{Circuit complexity} \label{Sec 2.2}
Let us recall the Boolean circuit complexity class $\DLOGTIME$-uniform $\ACz$ and related notions from \cite{vollmerIntroductionCircuitComplexity1999}. 

The class $\ACz$ consists of problems decidable by families of Boolean circuits of polynomial size and constant depth, using unbounded fan-in AND and OR gates, and NOT gates of fan-in one. A uniformity condition for a family $\{C_n: n \in \omega\}$ of Boolean circuits requires that the function $n \mapsto C_n$ can be easily computed in a way. In this paper, we will use \emph{$\DLOGTIME$-uniformity} from \cite{barringtonImmermanStraubing1990Uniformity} (see also \cite[Section 4.5.1]{vollmerIntroductionCircuitComplexity1999}), which is the standard uniformity notion for $\ACz$. We simply write uniform $\ACz$ for $\DLOGTIME$-uniform $\ACz$. Uniform $\ACz$ is strictly contained in $\ACz$; the latter even contains undecidable problems, such as $\{1^n: n \in A\}$ for an undecidable set $A \subseteq \omega$. Note that when dealing with (uniform) $\ACz$, we need to restrict to many-one reductions that are computable in (uniform) $\ACz$.

It follows immediately from the definition that uniform $\ACz$ is contained in $\sf{LSPACE}$ ($\subseteq \P$). This inclusion is strict, since the $\Parity$ problem, which asks whether a given binary string has an odd number of 1's, is in $\sf{LSPACE}$ but not in uniform $\ACz$. In fact, $\Parity$ is not even in $\ACz$ (see, e.g., \cite[Section 3.3.3]{vollmerIntroductionCircuitComplexity1999}). The same holds for the $\EvenParity$ problem, which asks whether a given binary string has an even number of 1's. We will show that certain problems are not in (uniform) $\ACz$ by reducing $\EvenParity$ to them. Other typical problems outside $\ACz$ include the majority problem and the modulo $n$ problems. These reflect the idea that $\ACz$ cannot ``count'' the number of 1's in a binary string.

Another fact we will use about uniform $\ACz$ is its connection to first-order logic. It is well-known that for a fixed first-order sentence in a finite language, the problem of deciding whether a given finite structure satisfies the sentence is in uniform $\ACz$. This is a consequence of the logical characterization of uniform $\ACz$ in terms of first-order logic (see, e.g., \cite[Theorem 4.73]{vollmerIntroductionCircuitComplexity1999}). Intuitively, we can build a Boolean circuit from a first-order sentence by turning universal quantifiers into AND gates and existential quantifiers into OR gates. Restricting to the language of Kripke frames, we obtain the following lemma.

\begin{lemma} \label{lem: first-order uAC0}
    Let $\alpha$ be a first-order sentence in the language $\L$ of Kripke frames. Then, the problem of deciding whether a given finite Kripke frame satisfies $\alpha$ is in uniform $\ACz$.
\end{lemma}

\section{Complexity of elementary modal logics} \label{Sec 3}

The aim of this section is to prove \Cref{Thm elementary in P}. It provides a necessary condition for a finitely axiomatizable logic to be elementary, namely, that the problem \Val{$L$} of deciding whether a given finite frame validates $L$ is in uniform $\ACz$. This will be used in the next section to show that certain $\Kf$-stable logics are not elementary, by showing that \Val{$L$} is not in $\ACz$. Formally, the validity problem \Val{$L$} for a logic $L$ is defined as follows.

\begin{definition}
    For a fixed logic $L$, the problem \Val{$L$} takes as input a finite Kripke frame $F$ and asks whether $F$ validates $L$.
\end{definition}

% Note that the size of an input frame $F$ is measured by the number of points and relational instances in $F$, but not by the cardinality $|F|$ of the underlying set. This is compatible with the fact that for decision problems for graphs, the input size is measured by the number of vertices and edges, rather than the cardinality of the vertex set.

The problem \Val{$\K$} is trivially in uniform $\ACz$, and the same holds for many standard logics such as $\Kf$ and $\Sf$. In fact, these are special cases of the following observation. Recall that a \emph{(global) first-order correspondent} of a modal formula $\phi$ is a first-order sentence $\alpha$ in the language $\L$ of Kripke frames such that for any Kripke frame $F$, we have $F \models \phi$ iff $F \models \alpha$ \cite{vanBenthem1976ModalCorrespondenceTheory} (see also \cite[Section 3.1]{blackburnModalLogic2001}).

\begin{proposition} \label{Prop fo-cor P}
    If a logic $L$ is axiomatized by a modal formula with a first-order correspondence, then \Val{$L$} is in uniform $\ACz$.
\end{proposition}

\begin{proof}
    Let $L = \K + \phi$ for some modal formula $\phi$ with a first-order correspondence $\alpha$. Then, for any finite Kripke frame $F$, we have $F \models L$ iff $F \models \phi$ iff $F \models \alpha$. Since $\alpha$ is a fixed first-order sentence, the problem of deciding whether $F \models \alpha$ is in uniform $\ACz$ by \Cref{lem: first-order uAC0}. Thus, \Val{$L$} is in uniform $\ACz$.
\end{proof}

In particular, it follows that \Val{$\K + \phi$} is in uniform $\ACz$ for any Sahlqvist formula $\phi$ \cite{CompletenessCorrespondenceFirst1975} (see also \cite[Section 3.6]{blackburnModalLogic2001}).

This phenomenon does not straightforwardly generalize to elementary logics. The elementarity of a logic $L$ only implies that $L$ is sound and complete with respect to \emph{some} elementary class $\class{K}$ of Kripke frames, which might not contain all finite frames validating $L$. So, the proof of \Cref{Prop fo-cor P} does not go through; in particular, the equivalence $F \models L$ iff $F \in \class{K}$ may not hold for finite Kripke frames. Another issue is that the elementary class $\class{K}$ might be axiomatized by infinitely many first-order sentences, which again invalidates the proof. Even though checking one first-order sentence can be done in uniform $\ACz$ by \Cref{lem: first-order uAC0}, it is not clear how to do this for infinitely many sentences. 

We fix these two issues by combining a trick used in \cite{GoldblattHodkinsonVenema2004Erdos} with a compactness argument to show that if a logic $L$ is finitely axiomatizable and elementary, then \Val{$L$} is in uniform $\ACz$ (\Cref{Thm elementary in P}). The idea is as follows. Intuitively, \cite[Proposition 2.17]{GoldblattHodkinsonVenema2004Erdos} states that if a logic $L$ is elementary, then $L$ is in fact sound and complete with respect to a ``nice'' elementary class $\class{K}$ of Kripke frames that contains all finite frames validating $L$. Then, using a compactness argument with the method of standard translations (see, e.g., \cite[Section 2.4]{blackburnModalLogic2001}), we show that if $L$ is in addition finitely axiomatizable, then we can take an even ``nicer'' class $\class{K}$ that is axiomatizable by finitely many first-order sentences. Therefore, we may apply \Cref{lem: first-order uAC0} and conclude that \Val{$L$} is in uniform $\ACz$.

The following fact is a reformulation of \cite[Proposition 2.17]{GoldblattHodkinsonVenema2004Erdos} in terms of unimodal logics, which is a consequence of \cite[Theorem 3.6.7]{Goldblatt1989Varieties} and \cite[Theorem 4.12]{Goldblatt1995Elementary}. It was used in \cite{GoldblattHodkinsonVenema2004Erdos} to show that certain canonical logics are not elementary, thus refuting the converse of Fine's theorem \cite{ConnectionsElementaryModal1975}.

\begin{proposition} \label{Prop goldblatt}
    A logic $L$ is elementary iff there is an elementary class $\class{K}$ of Kripke frames satisfying $\CSt(L) \subseteq \class{K} \subseteq \KF(L)$.
\end{proposition}

Now we prove the main theorem of this section with an intermediate lemma. 
% For a class $\class{K}$ of Kripke frames, we write $\class{K}\fin$ for the class of all finite Kripke frames in $\class{K}$.

\begin{lemma} \label{Lem exists K}
    If $L$ is a finitely axiomatizable elementary logic, then there exists a finitely axiomatizable elementary class $\class{K}$ of Kripke frames such that $\class{K} \subseteq \KF(L)$ and $\class{K}\fin = \KF(L)\fin$.
\end{lemma}

\begin{proof}
    Let $L$ be a finitely axiomatizable elementary logic. Without loss of generality, we may assume that $L = \K +\phi$ for some modal formula $\phi$. By \Cref{Prop goldblatt}, there is an elementary class $\class{K}$ of Kripke frames such that $\CSt(L) \subseteq \class{K} \subseteq \KF(L)$. 
    
    First, we show that $\class{K}\fin = \KF(L)\fin$. Since $\class{K} \subseteq \KF(L)$, we have $\class{K}\fin \subseteq \KF(L)\fin$. Conversely, let $F \in \KF(L)\fin$. Then, $F \models L$, so $F^+$ is a finite modal algebra validating $L$, and thus $(F^+)_+ \in \CSt(L)$. Since $F$ is finite, it is isomorphic to its ultrafilter extension $(F^+)_+$ (see, e.g., \cite[Section 2.5]{blackburnModalLogic2001}), and thus $F \in \CSt(L) \subseteq \class{K}$, which implies $F \in \class{K}\fin$ since $F$ is finite. Therefore, $\class{K}\fin = \KF(L)\fin$.
    
    Recall that $\L$ is the first-order language of Kripke frames. For a first-order theory $T'$, we write $\Mod(T')$ for the class of structures satisfying $T'$. Let $T$ be a possibly infinite first-order theory in $\L$ axiomatizing $\class{K}$, that is, $\Mod(T) = \class{K}$. We show that if there is a finite subset $T_0 \subseteq T$ such that $\Mod(T_0) \subseteq \KF(L)$, then $\Mod(T_0)$ is the desired class of Kripke frames. Clearly, $\Mod(T_0)$ is a finitely axiomatizable elementary class. Since $T_0 \subseteq T$, we have $\Mod(T) \subseteq \Mod(T_0)$, and thus $\KF(L)\fin = \class{K}\fin = \Mod(T)\fin \subseteq \Mod(T_0)\fin$. For the other inclusion, we have $\Mod(T_0) \subseteq \KF(L)$ by assumption, which implies $\Mod(T_0)\fin \subseteq \KF(L)\fin$, and hence $\Mod(T_0)\fin = \KF(L)\fin$.
    
    Finally, we show that such a finite subset $T_0 \subseteq T$ indeed exists by a compactness argument. Suppose for a contradiction that no finite subset $T_0 \subseteq T$ satisfies $\Mod(T_0) \subseteq \KF(L)$. Then, for each finite subset $T_0 \subseteq T$, there is a Kripke frame $F_{T_0} \in \Mod(T_0)$ such that $F_{T_0} \not\in \KF(L)$, which means that $F_{T_0} \not\models \phi$. Let $\L_\phi$ be the first-order language extending $\L$ with an additional predicate symbol for each propositional variable occurring in $\phi$. A Kripke model $(F, V)$ considered in the context of $\phi$, with the valuation $V$ restricted to the propositional variables occurring in $\phi$, can naturally be regarded as an $\L_\phi$-structure by interpreting each additional predicate symbol $P$ as the truth set $V(p)$ of the corresponding propositional variable $p$ in $\phi$. Let $ST_x(\phi)$ be the standard translation of $\phi$, a first-order formula in $\L_\phi$ with free variable $x$, and $\alpha = \forall x ST_x(\phi)$. Then, $\alpha$ is a first-order sentence in $\L_\phi$ such that a Kripke model validates $\phi$ iff it satisfies $\alpha$. Since $F_{T_0} \not\models \phi$, there is a valuation $V_{T_0}$ on $F_{T_0}$ over the propositional variables occurring in $\phi$ such that $(F_{T_0}, V_{T_0}) \not\models \phi$. Let $M_{T_0}$ be the Kripke model $(F_{T_0}, V_{T_0})$. Then, $M_{T_0} \not\models \alpha$, while $M_{T_0} \models T_0$ since its underlying frame $F_{T_0} \models T_0$. Therefore, the $\L_\phi$-theory $T_0 \cup \{\lnot \alpha\}$ is satisfiable, for each finite subset $T_0 \subseteq T$. By the compactness theorem for first-order logic, the $\L_\phi$-theory $T \cup \{\lnot \alpha\}$ is satisfiable. However, this contradicts $\class{K} \subseteq \KF(L)$. Indeed, if $M$ is an $\L_\phi$-structure such that $M \models T \cup \{\lnot \alpha\}$, then $M \models T$ and $M \not\models \phi$. For the underlying frame $F$ of $M$, the former implies $F \in \Mod(T)$ since $T$ is an $\L$-theory and $F$ is the $\L$-reduct of $M$, and the latter implies $F \not\models \phi$. These contradict $\class{K} = \Mod(T) \subseteq \KF(L)$. Thus, there is a finite subset $T_0 \subseteq T$ such that $\Mod(T_0) \subseteq \KF(L)$, completing the proof.
\end{proof}

\begin{remark}
    \Cref{Lem exists K} is a useful variant of \Cref{Prop goldblatt} (\cite[Proposition 2.17]{GoldblattHodkinsonVenema2004Erdos}). Moreover, the conclusion of \Cref{Lem exists K}, if we drop the condition $\class{K} \subseteq \KF(L)$, was called \emph{finitely elementary} and studied in \cite{balbianiEveryWorldCan}. We note that the converse of \Cref{Lem exists K} does not hold. The logics $\logic{KM^\infty}$ and $\logic{KMT}$ mentioned in the introduction serve as counterexamples. Hughes \cite{hughesEveryWorldCan1990} showed that $\logic{KMT}$ satisfies the conclusion of \Cref{Lem exists K} by taking $\class{K}$ to be the class of all Kripke frames satisfying $\forall x \exists y (xRy \land yRy)$, but $\logic{KMT}$ is not finitely axiomatizable. Similarly, $\logic{KM^\infty}$ satisfies the conclusion of \Cref{Lem exists K} by taking $\class{K}$ to be the class of all Kripke frames satisfying 
    \[\forall x \exists y (xRy \land \forall z \forall z' (yRz \land yRz' \to z = z')),\]
    which follows from \cite[Theorem 7]{balbianiEveryWorldCan} (the first-order sentence was identified by Lemmon \cite{LemmonNotes} to prove the Kripke completeness), but $\logic{KM^\infty}$ is not finitely axiomatizable \cite{goldblattMcKinseyLemmonLogic2007}.
\end{remark}

\begin{theorem} \label{Thm elementary in P}
    If $L$ is a finitely axiomatizable elementary logic, then \Val{$L$} is in uniform $\ACz$.
\end{theorem}

\begin{proof}
    Let $L$ be a finitely axiomatizable elementary logic. By \Cref{Lem exists K}, there is a finitely axiomatizable elementary class $\class{K}$ of Kripke frames such that $\class{K} \subseteq \KF(L)$ and $\class{K}\fin = \KF(L)\fin$. Let $\alpha$ be a first-order sentence axiomatizing $\class{K}$. Then, for each finite Kripke frame $F$, we have $F \models L$ iff $F \in \KF(L)\fin$ iff $F \in \class{K}\fin$ iff $F \models \alpha$. Since $\alpha$ is a fixed first-order sentence, the problem of deciding whether $F \models \alpha$ is in uniform $\ACz$ by \Cref{lem: first-order uAC0}. Thus, \Val{$L$} is in uniform $\ACz$.
\end{proof}

\begin{remark}
    Other than our main application in \Cref{Sec 4}, \Cref{Thm elementary in P} can be used to show that some (union-)splittings in $\NExt{\K}$ are not elementary. Splittings and union-splittings are defined lattice-theoretically, but they are axiomatized by Jankov formulas (or splitting formulas) $\J(F)$ for finite rooted cycle-free Kripke frames $F$, which have semantic characterizations similar to stable formulas in terms of p-morphisms. We refer the reader to \cite[Section 10.5]{czModalLogic1997} and \cite[Section 7]{stablecanonicalrules} for precise definitions and details. 
    
    However, Rafter \cite{rafterPartialCharacterizationCanonical1994} already discovered (union-)splittings that are not canonical, which by Fine's theorem implies that they are not elementary. Thus, we keep this application as a concise remark. 

    By reducing \Parity{} for nonempty inputs in a similar way as in \Cref{Sec 4}, one can show that the validity problem \Val{$\K + \J(F)$}, where $F$ is the Kripke frame depicted in \Cref{Fig F}, is not in $\ACz$. Thus, $\K + \J(F)$ is not elementary by \Cref{Thm elementary in P}.

    \begin{figure}[htb]
    \centering
    \begin{tikzpicture}
    \node[] (r) at (0,0) {$\bullet$};
    \node[] (a) at (0,1) {$\bullet$};
    \node[] (x) at (-1,2) {$\bullet$};
    \node[] (y) at (1,2) {$\bullet$};

    \draw[->] (r) -- (a);
    \draw[->] (a) -- (x);
    \draw[->] (a) -- (y);
    \draw[->] (r) -- (x);

\end{tikzpicture}
    \caption{The Kripke frame $F$.}
    \label{Fig F}
\end{figure}

\end{remark}

\section{Non-elementary \texorpdfstring{$\Kf$}{K4}-stable logics} \label{Sec 4}
In this section, we show the existence of non-elementary $\Kf$-stable logics. More specifically, we show that for certain $\Kf$-stable logics $L$, the problem \Val{$L$} is not in $\ACz$, which by the contraposition of \Cref{Thm elementary in P} implies that $L$ is not elementary. We will only consider $\Kf$-frames, namely, transitive Kripke frames, so in the drawings we omit the relations implied by transitivity. 

First, let us recall stable logics and stable formulas from \cite{stablemodallogic} (see also \cite{FiltrationRevisitedLattices2018} for a more comprehensive account). We follow the relational approach instead of the algebraic approach since it is more convenient for combinatorial arguments. A \emph{modal space}, also known as a \emph{descriptive frame}, is a pair $\X = (X, R)$ of a Stone space and a binary relation $R \subseteq X \times X$ such that $R[x]$ is closed for each $x \in X$ and the set $R^{-1}[U]$ is clopen for each clopen set $U \subseteq X$. A continuous map $f: \X \to \Y = (Y, Q)$ between modal spaces is \emph{stable} if it is relation-preserving, i.e., $xRy$ implies $f(x)Qf(y)$ for all $x, y \in X$. If $f: \X \to \Y$ is a surjective stable map, we call $\Y$ a \emph{stable image} of $\X$. Finite modal spaces are identified with finite Kripke frames with the discrete topology, so any map from a finite Kripke frame is always continuous. Let $M$ be a modal logic. A class of $M$-spaces is \emph{$M$-stable} if it is closed under stable images that validate $M$. A modal logic $L \supseteq M$ is \emph{$M$-stable} if $L$ is sound and complete with respect to an $M$-stable class of $M$-spaces. 

Stable formulas are a special type of \emph{stable canonical formulas} \cite{stablecanonicalrules,stablemodallogic}, and are designed to axiomatize stable logics. Recall that a $\Kf$-frame is simply a transitive Kripke frame. Each finite rooted $\Kf$-frame $F$ induces a stable formula $\gamma(F, \emp)$, which has the property that for any $\Kf$-space $\X$, we have $\X \not\models \gamma(F, \emp)$ iff $F$ is a stable image of a topo-rooted closed upset of $\X$ \cite[Theorem 6.8]{stablecanonicalrules}. In particular, restricting to finite $\Kf$-frames, we obtain the following fact. Recall that a subset $U$ of a Kripke frame $G = (G, R)$ is an \emph{upset} if $x \in U$ and $xRy$ implies $y \in U$.

\begin{proposition} \label{Prop stable formulas}
    Let $\gamma(F, \emp)$ be a stable formula where $F$ is a finite rooted $\Kf$-frame. Then, for any finite $\Kf$-frame $G$, we have 
    \[G \not\models \gamma(F, \emp) \text{ iff $F$ is a stable image of a rooted upset of $G$}.\]
\end{proposition}

This characterization of stable formulas will be particularly useful for our purpose. As we will see, it allows us to encode the (even-)parity problem in the validity problem for logics axiomatized by stable formulas. Although all $\Kf$-stable logics are axiomatized by stable formulas \cite[Theorem 4.7]{stablemodallogic}, the converse does not hold in general \cite[Example 4.11]{stablemodallogic}. But the stable formulas $\gamma(F, \emp)$ we will consider are based on $F$ with a reflexive root. These stable formulas indeed axiomatize stable logics, by the following fact \cite[Proposition 4.13]{stablemodallogic}.

\begin{lemma} \label{Lem K4-stable}
    Let $L = \Kf + \{\gamma(F_i, \emp): i \in I\}$, where each $F_i$ is a finite rooted $\Kf$-frame with a reflexive root. Then $L$ is $\Kf$-stable.
\end{lemma}

Now we are ready to define an infinite family of finite rooted $\Kf$-frames $F_n$ ($n \geq 1$) and show that the logics $\Kf + \gamma(F_n, \emp)$ are not elementary. 

\begin{definition}
    For $n \geq 1$, let $F_n = (F_n, R_n)$ be the finite rooted $\Kf$-frame depicted in \Cref{Fig Fn}.

    \begin{figure}[htb]
    \centering
    \begin{tikzpicture}
    \node[label={left:$c$}] (c) at (0,5.5) {$\bullet$};
    \node[label={above:$x'$}] (x') at (2,5.5) {$\circ$};
    \node[label={above:$y'$}] (y') at (4,5.5) {$\circ$};

    \node[label={left:$x$}] (x) at (2,3.5) {$\circ$};
    \node[label={right:$y$}] (y) at (4,3.5) {$\circ$};

    \node[label={right:$b$}] (b) at (6,3.5) {$\bullet$};

    \node[label={left:$a_1$}] (a1) at (3,2.6) {$\bullet$};
    \node (chain) at (3,1.8) {\rotatebox{90}{$\dots$}};
    \node[label={left:$a_n$}] (an) at (3,1) {$\bullet$};
    \node[label={left:$r$}] (r) at (3,0) {$\circ$};

    \node at (3,-1) {$F_n$};

    \draw[->] (x) -- (c);
    \draw[->] (x) -- (x');

    \draw[->] (y) -- (c);
    \draw[->] (y) -- (y');

    \draw[->] (b) -- (x');
    \draw[->] (b) -- (y');

    \draw[->] (a1) -- (x);
    \draw[->] (a1) -- (y);
    \draw[->] (an) -- (chain);
    \draw[->] (chain) -- (a1);

    \draw[->] (r) -- (an);
    \draw[->] (r) -- (b);
\end{tikzpicture}
    \caption{The Kripke frame $F_n$. Relations obtained by transitivity are omitted.}
    \label{Fig Fn}
\end{figure}

    Formally, we define
    \begin{align*}
        F_n &= \{r, b, c\} \cup \{a_i: 1 \leq i \leq n\} \cup \{x, y, x', y'\}, \\
        R_n &\text{ is the transitive closure of } \\
        &\{(r, r), (r, a_n), (r, b)\} \\
        &\cup \{(a_{i+1}, a_i): 1 \leq i < n\} \cup \{(a_1, x), (a_1, y)\} \\
        &\cup \{(x, x), (y, y), (x, c), (y, c), (x, x'), (y, y')\} \\
        &\cup \{(x', x'), (y', y'), (b, x'), (b, y')\}.
    \end{align*}
\end{definition}

\begin{lemma} \label{lem: K4-stable not in AC0}
    For $n \geq 1$, the problem \Val{$\Kf + \gamma(F_n, \emp)$} is not in $\ACz$.
\end{lemma}

\begin{proof}
    Let $n \geq 1$. We reduce \EvenParity{} to \Val{$\Kf + \gamma(F_n, \emp)$} via a uniform $\ACz$ reduction. Then, the statement follows from the fact that \EvenParity{} is not in $\ACz$. 

    Given $s \in \{0, 1\}^m$ ($m \geq 0$) as an input for \EvenParity{}, we construct a finite rooted $\Kf$-frame $G_s = (G_s, R_s)$ by defining
    \begin{align*}
        G_s = &\{r, b, c\} \cup \{a_i: 1 \leq i \leq n\} \cup W_s \cup W'_s, \\
        &\text{ where } W_s = \{u_i, v_i: 0 \leq i \leq m\} \text{ and }W'_s = \{u'_i, v'_i: 0 \leq i \leq m\}, \\
        R_s = &\{(r, w): w \in G_s\} \\
        &\cup \{(a_i, a_j): 1 \leq j < i \leq n\} \cup \{(a_i, w): 1 \leq i \leq n, w \in W_s \cup W'_s \cup \{c\}\} \\
        &\cup \{(w, w): w \in W_s \cup W'_s\} \\
        &\cup \{(w, c): w \in W_s\} \cup \{(b, w): w \in W'_s\} \\
        &\cup \{(u_i, u'_i), (v_i, v'_i): 0 \leq i \leq m\} \\
        &\cup \{(u_i, u'_{i+1}), (v_i, v'_{i+1}): 0 \leq i < m, s(i) = 0\} \\
        &\cup \{(u_i, v'_{i+1}), (v_i, u'_{i+1}): 0 \leq i < m, s(i) = 1\} \\
        &\cup \{(u_0, v'_m)\}.
    \end{align*}
    The construction replaces the $\{x, y, x', y'\}$ part of $F_n$ with a subframe $W_s \cup W'_s$. The central idea is that the parity of $s$ is even iff $(u_0, v'_m) \in R_s$ connects the two connected components in $W_s \cup W'_s$ containing $u'_0$ and $v'_0$. For example, for $s = 00$ and $s = 10$, \Cref{Fig Gs} illustrates the subframe $W_s \cup W'_s$ of $G_s$. 
    \begin{figure}[htb]
    \centering
    \begin{tikzpicture}
    \begin{scope}
        \foreach \i in {0,1,2} {
            \node[label={below:$u_\i$}] (u\i) at (2*\i,0) {$\circ$};
            \node[blue,label={[text=blue]below:$v_\i$}] (v\i) at (2*\i+1,0) {$\circ$};
            \node[label={above:$u'_\i$}] (up\i) at (2*\i,2.5) {$\circ$};
            \node[blue,label={[text=blue]above:$v'_\i$}] (vp\i) at (2*\i+1,2.5) {$\circ$};
            \draw[->] (u\i) -- (up\i);
            \draw[->,blue] (v\i) -- (vp\i);
        }
        \node at (2.5,-1) {$G_{00}$};
        \draw[->] (u0) -- (up1);
        \draw[->,blue] (v0) -- (vp1);
        \draw[->] (u1) -- (up2);
        \draw[->,blue] (v1) -- (vp2);
        
        \draw[->] (u0) -- (vp2);
    \end{scope}

    \begin{scope}[xshift=7.2cm]
        \foreach \i/\ucolor/\vcolor in {0/black/blue,1/blue/black,2/blue/black} {
            \node[\ucolor,label={[text=\ucolor]below:$u_\i$}] (u\i) at (2*\i,0) {$\circ$};
            \node[\vcolor,label={[text=\vcolor]below:$v_\i$}] (v\i) at (2*\i+1,0) {$\circ$};
            \node[\ucolor,label={[text=\ucolor]above:$u'_\i$}] (up\i) at (2*\i,2.5) {$\circ$};
            \node[\vcolor,label={[text=\vcolor]above:$v'_\i$}] (vp\i) at (2*\i+1,2.5) {$\circ$};
            \draw[->,\ucolor] (u\i) -- (up\i);
            \draw[->,\vcolor] (v\i) -- (vp\i);
        }
        \node at (2.5,-1) {$G_{10}$};
        \draw[->] (u0) -- (vp1);
        \draw[->,blue] (v0) -- (up1);
        \draw[->,blue] (u1) -- (up2);
        \draw[->] (v1) -- (vp2);
        \draw[->] (u0) -- (vp2);
    \end{scope}
\end{tikzpicture}
    \caption{The subframe $W_s \cup W'_s$ of $G_{00}$ and $G_{10}$.}
    \label{Fig Gs}
\end{figure}

    Note that $R_s$ is transitive; we may not take a transitive closure since it is not computable in $\ACz$. It is standard to verify that the construction $s \mapsto G_s$ is computable in $\ACz$. Intuitively, this is because the instances of $R_s$ are defined locally without counting or recursion. Thus, it remains to show the correctness of the reduction.

    \begin{claim}
        For $s \in \{0, 1\}^m$ ($m \geq 0$), the parity of $s$ is even iff $G_s \models \gamma(F_n, \emp)$.
    \end{claim}

    \begin{proof}
        First, we color the points in $W_s \cup W'_s$ as follows. Points $u_0, u'_0$ are black, and points $v_0, v'_0$ are blue. For $0 \leq i < m$, we extend the coloring according to $s(i)$, reflecting the definition of $R_s$. That is, if $s(i) = 0$, then $u_{i+1}, u'_{i+1}$ have the same color as $u_i, u'_i$, and $v_{i+1}, v'_{i+1}$ as $v_i, v'_i$. Similarly, if $s(i) = 1$, then $v_{i+1}, v'_{i+1}$ have the same color as $u_i, u'_i$, and $u_{i+1}, u'_{i+1}$ as $v_i, v'_i$. This coloring is shown for $G_{00}$ and $G_{10}$ in \Cref{Fig Gs}.

        Then, one can inductively verify that (1) every $R_s$ relation except for $(u_0, v'_m)$ is between points of the same color, (2) every two points of the same color are connected by an $R_s \cup (R_s)^{-1}$-path in $W_s \cup W'_s$, and (3) $v'_m$ is black iff the parity of $s$ is odd. 
        
        Now, suppose that the parity of $s$ is odd. Then, $u_0$ and $v'_m$ are both black by (3), so all $R_s$ relations are between points of the same color by (1). We define the map $f: G_s \to F_n$ by
        \begin{equation*}
            f(w) = \begin{cases}
                w & \text{if } w \in \{r, b, c\} \cup \{a_i: 1 \leq i \leq n\}, \\
                x & \text{if } w \in W_s \text{ and $w$ is black}, \\
                x' & \text{if } w \in W'_s \text{ and $w$ is black}, \\
                y & \text{if } w \in W_s \text{ and $w$ is blue}, \\
                y' & \text{if } w \in W'_s \text{ and $w$ is blue}.
            \end{cases}
        \end{equation*} 
        The map $f$ is surjective since $u_0, u'_0$ are black and $v_0, v'_0$ are blue. To show that $f$ is stable, it suffices to observe that there is no $R_s$ relation from a point in $W'_s$ to a point in $W_s$, nor between points of different colors in $W_s \cup W'_s$. Thus, $f$ is a surjective stable map from $G_s$, which is a rooted upset of $G_s$ itself, to $F_n$, and hence $G_s \not\models \gamma(F_n, \emp)$ by \Cref{Prop stable formulas}.
        
        Conversely, suppose that the parity of $s$ is even. Then, $u_0$ and $v'_m$ are of different colors by (3), so every two points in $W_s \cup W'_s$ are connected by an $R_s \cup (R_s)^{-1}$-path in $W_s \cup W'_s$ by (2). Assume for a contradiction that $G_s \not\models \gamma(F_n, \emp)$. Then, there is a surjective stable map $f: U \to F_n$ for some rooted upset $U$ of $G_s$ by \Cref{Prop stable formulas}. Since a stable map cannot map a reflexive point to an irreflexive point, we have $\{b, c\} \cup \{a_i: 1 \leq i \leq n\} \subseteq U$, so $U = G_s$ because it is rooted. Similarly, since $f$ is stable and surjective, we observe that $f(w) = w$ for $w \in \{r, b, c\} \cup \{a_i: 1 \leq i \leq n\}$. 
        
        Thus, $f {\upharpoonright} (W_s \cup W'_s)$ is a surjective stable map from $W_s \cup W'_s$ to $\{x, y, x', y'\}$. However, for $w, w' \in W_s \cup W'_s$ such that $f(w) = x$ and $f(w') = y'$, there is an $R_s \cup (R_s)^{-1}$-path in $W_s \cup W'_s$ connecting them, which contradicts that $f$ is stable and there is no $R_n \cup (R_n)^{-1}$-path in $\{x, y, x', y'\}$ connecting $x$ and $y'$. Therefore, such a map $f$ cannot exist, and hence $G_s \models \gamma(F_n, \emp)$.
    \end{proof}

    The above claim shows the correctness of the reduction, and thus completes the proof.
\end{proof}

\begin{theorem} \label{Thm K4-stable non-elem}
    There are infinitely many finitely axiomatizable $\Kf$-stable logics that are not elementary.
\end{theorem}

\begin{proof}
    Let $n \geq 1$. The logic $\Kf + \gamma(F_n, \emp)$ is finitely axiomatizable, and is $\Kf$-stable by \Cref{Lem K4-stable} since $F_n$ has a reflexive root. If it were elementary, then \Cref{Lem exists K} would give a fixed first-order sentence defining its finite validating frames. By \Cref{lem: first-order uAC0}, its validity problem would be in uniform $\ACz$, contradicting \Cref{lem: K4-stable not in AC0}.

    Finally, these logics are pairwise distinct. For $n < n'$, the frame $F_n$ refutes $\gamma(F_n, \emp)$, while it validates $\gamma(F_{n'}, \emp)$ because $|F_n| < |F_{n'}|$, so $F_{n'}$ cannot be a stable image of any rooted upset of $F_n$.
\end{proof}

Recall the pre-transitive logics $\Kff{m+1}{1} = \K + \Box p \to \Box^{m+1} p$ for $m \geq 1$. These logics define the \emph{$(m+1, 1)$-transitivity}: $\forall x_0 \cdots x_{m+1} (x_0 R x_1 \land \cdots \land x_mRx_{m+1} \to x_0Rx_{m+1})$. The logic $\Kff{2}{1}$ is exactly $\Kf$. The theory of stable canonical formulas was generalized to these logics in \cite{takahashi2025stablecanonicalrulesformulas}, yielding the finite model property of $\Kff{m+1}{1}$-stable logics for $m \geq 1$. All $\Kff{m+1}{1}$-stable logics are axiomatizable by stable formulas $\gamma^m(F, \emp)$ designed for $\Kff{m+1}{1}$. Moreover, \Cref{Prop stable formulas} and \Cref{Lem K4-stable} have their analogues for $\Kff{m+1}{1}$ \cite[Theorem 5.5 and Lemma 5.15, respectively]{takahashi2025stablecanonicalrulesformulas}. Since the Kripke frames $F_n$ and $G_s$ are transitive, they are in particular $\Kff{m+1}{1}$-frames for all $m \geq 1$. So, the same argument in this section applies to these pre-transitive settings. It follows that, for $m \geq 1$ and $n \geq 1$, the logics $\Kff{m+1}{1} + \gamma^m(F_n, \emp)$ are $\Kff{m+1}{1}$-stable and not elementary.

Not only are there infinitely many such logics $\Kff{m+1}{1} + \gamma^m(F_n, \emp)$ by a similar argument as in the proof of \Cref{Thm K4-stable non-elem}, but also $\Kff{m+1}{1} + \gamma^m(F_n, \emp)$ does not extend $\Kff{m'+1}{1}$ for $m' < m$. Indeed, let $C$ be an irreflexive chain that is $(m+1, 1)$-transitive but not $(m'+1, 1)$-transitive. For instance, we can take $C = (C, R)$ where $C = \{0, 1, \ldots, m'+1\}$ and $R = \{(i, j) : j = i+1\}$, which is $(m+1, 1)$-transitive because there is no path of length $m+1$ by $m' < m$, and is clearly not $(m'+1, 1)$-transitive. It is readily verified that $F_n$ is not a stable image of any upset of $C$, since there is no path in $F_n$ that contains all points of $F_n$. So, we have $C \models \Kff{m+1}{1} + \gamma^m(F_n, \emp)$ and $C \not\models \Kff{m'+1}{1}$. In summary, we obtain the following result, with \Cref{Thm K4-stable non-elem} as the special case of $m=1$.

\begin{theorem}
    For each $m \geq 1$, there are infinitely many finitely axiomatizable $\Kff{m+1}{1}$-stable logics that are not elementary and are not $\Kff{m'+1}{1}$-stable for $m' < m$.
\end{theorem}

We conclude this section by noting that our method does not apply straightforwardly to $\K$-stable logics and $\Sf$-stable logics for different reasons.

\begin{remark}
    For $\K$-stable logics, the issue is that we do not have a counterpart of stable formulas. The best we can do is to use \emph{stable rules}, which are multi-conclusion inference rules that axiomatize all $\K$-stable logics \cite{stablecanonicalrules}. However, axiomatizations of logics via rules are weaker than axiomatizations via formulas, in the sense that a Kripke frame may validate the logic without validating all rules axiomatizing the logic. Algebraically, axiomatizations via formulas characterize the variety corresponding to the logic, whereas axiomatizations via rules only characterize a universal class that generates the variety. Therefore, the very first step of the proof in \Cref{lem: K4-stable not in AC0}, where we implicitly used the fact that $F \models \Kf + \gamma(F_n, \emp)$ iff $F \models \Kf$ and $F \models \gamma(F_n, \emp)$, does not work for axiomatizations of $\K$-stable logics via rules.
\end{remark} 

\begin{remark}
    For $\Sf$-stable logics, the obstruction is very clear. Our construction of $F_n$ and the reduction make an essential use of irreflexive points, which are not allowed in $\Sf$-frames. These irreflexive points, together with the fact that a stable map cannot map a reflexive point to an irreflexive point, were needed to control the potential stable maps. 
\end{remark}

\section{Conclusion}
In this paper, we showed that the validity problem for finitely axiomatizable elementary logics is in uniform $\ACz$. Applying its contrapositive, we obtained that there are infinitely many finitely axiomatizable non-elementary $\Kff{m+1}{1}$-stable logics for each $m \geq 1$.

Since there are only countably many finitely axiomatizable logics, this result shows that there are as many as possible finitely axiomatizable non-elementary $\Kff{m+1}{1}$-stable logics. On the other hand, the existence of non-canonical $\Kff{m+1}{1}$-stable logics remains open (see also \cite[Table 4.8.1]{FiltrationRevisitedLattices2018}). We conjecture that the logics we constructed in \Cref{Sec 4} are not canonical.

Another potential application of our method is to show the existence of finitely axiomatizable canonical logics that are not elementary, providing a finitely axiomatizable counterexample to the converse of Fine's theorem. A counterexample was constructed in \cite{GoldblattHodkinsonVenema2004Erdos}, where the finitely axiomatizable case was left open \cite[Problem 4.2]{GoldblattHodkinsonVenema2004Erdos}. The circuit complexity approach can be used to show non-elementarity for finitely axiomatizable logics without showing non-canonicity.

\section*{Acknowledgements}
I am very grateful to Emil Je\v r\'abek, who pointed out to me the connection between first-order logic and circuit complexity, which helped me to remove the $\P \neq \NP$ assumption from my original proof. I would also like to thank Ian Hodkinson and Nick Bezhanishvili for their valuable comments on an earlier version of this paper, and Dmitry Shkatov for insightful discussions. I acknowledge the support of the Student Exchange Support Program (Graduate Scholarship for Degree Seeking Students) of the Japan Student Services Organization.

\section*{Use of AI}
This paper was written by the author with the help of inline text completion by GitHub Copilot. GPT-6 Astra and GPT-6.1 Sol were used for review and drawing assistant. The author takes full responsibility for the correctness and content of the manuscript.

\printbibliography

\end{document}